\documentclass[a4paper,10pt,leqno,twoside]{tobart}
\usepackage[english]{babel} 
\usepackage{inputenc, amsmath, amssymb, latexsym, epsfig, rotating, fancyheadings, amsthm}

\newcommand{\equi}{\ensuremath{\Leftrightarrow}}

\newcommand{\ldot}{\ensuremath{\textbf{.}}}

\newcommand{\ssq}{\ensuremath{\subseteq}}

\newcommand{\eps}{\ensuremath{\varepsilon}}

\newcommand{\T}{\ensuremath{\mathbb{T}}}

\newcommand{\N}{\ensuremath{\mathbb{N}}} 
\newcommand{\R}{\ensuremath{\mathbb{R}}}
\newcommand{\Z}{\ensuremath{\mathbb{Z}}}

\newcommand{\Id}{\ensuremath{\mathrm{Id}}}

\newcommand{\kreis}{\ensuremath{\mathbb{T}^{1}}}

\newcommand{\alphlist}{\begin{list}{(\alph{enumi})}{\usecounter{enumi}\setlength{\parsep}{2pt}
      \setlength{\itemsep}{1pt} \setlength{\topsep}{5pt}
      \setlength{\partopsep}{3pt}}}

\newcommand{\arablist}{\begin{list}{(\arabic{enumi})}{\usecounter{enumi}\setlength{\parsep}{2pt}
          \setlength{\itemsep}{1pt} \setlength{\topsep}{5pt}
          \setlength{\partopsep}{3pt}}}

\newcommand{\romanlist}{\begin{list}{(\roman{enumi})}{\usecounter{enumi}\setlength{\parsep}{2pt}
              \setlength{\itemsep}{1pt} \setlength{\topsep}{5pt}
              \setlength{\partopsep}{3pt}}}

 \newcommand{\listend}{\end{list}}

\newcommand{\bulletlist}{\begin{list}{$\bullet$}{\setlength{\parsep}{2pt}
                \setlength{\itemsep}{1pt} \setlength{\topsep}{5pt}
                \setlength{\partopsep}{3pt}\setlength{\leftmargin}{15pt}}}

\newcommand{\nLim}{\ensuremath{\lim_{n\rightarrow\infty}}}

\newcommand{\tLim}{\ensuremath{\lim_{t\rightarrow\infty}}}

\newcommand{\inergsum}{\ensuremath{\sum_{i=0}^{n-1}}}

\newcommand{\ntel}{\ensuremath{\frac{1}{n}}}

\newtheoremstyle{tobthm}{3pt}{3pt}{\itshape}{0pt}{\bfseries}{.}{0.5eM}{}
\theoremstyle{tobthm}

\newtheorem{definition}{Definition}[section]
\newtheorem{thm}[definition]{Theorem}

\newtheorem{lem}[definition]{Lemma}
\newtheorem{cor}[definition]{Corollary}  
\newtheorem{prop}[definition]{Proposition}

\newtheorem{questions}[definition]{Questions}

\newtheoremstyle{tobrem}{3pt}{3pt}{\normalfont}{0pt}{\bfseries}{.}{0.5em}{}
\theoremstyle{tobrem}

\newtheorem{rem}[definition]{Remark}

\numberwithin{equation}{section}
\numberwithin{figure}{section}

\title{\Large\textsc{\titlename}}
\author{J.~Aliste-Prieto\thanks{Centro de Modelamiento Matem\'atico,
    Universidad de Chile, Santiago. Email: {\tt jaliste@dim.uchile.cl}} and
  T.~J\"ager\thanks{Institute for Analysis, TU Dresden, Germany.
    Email: {\tt Tobias.Oertel-Jaeger@tu-dresden.de}}}

\newcommand{\titlename}{Almost periodic structures and the
  semiconjugacy problem
}

\begin{document}

\setlength{\abovedisplayskip}{1.0ex}
\setlength{\abovedisplayshortskip}{0.8ex}

\setlength{\belowdisplayskip}{1.0ex}
\setlength{\belowdisplayshortskip}{0.8ex}

\maketitle 

\abstract{The description of almost periodic or quasiperiodic structures has a long
  tradition in mathematical physics, in particular since the discovery
  of quasicrystals in the early 80's.  Frequently, the modelling of
  such structures leads to different types of dynamical systems which
  include, depending on the concept of quasiperiodicity being considered,
  skew products over quasiperiodic or almost-periodic base flows,
  mathematical quasicrystals or maps of the real line with
  almost-periodic displacement. An important problem in this context
  is to know whether the considered system is semiconjugate to a rigid
  translation. We solve this question in a general setting that
  includes all the above-mentioned examples and also allows to treat
  scalar differential equations that are almost-periodic both in space
  and time. To that end, we study a certain class of flows that
  preserve a one-dimensional foliation and show that a semiconjugacy
  to a minimal translation flow exists if and only if a boundedness
  condition, concerning the distance of orbits of the flow to
  those of the translation, holds.  }

\section{Introduction}

A classical topic in the theory of dynamical systems is the study of
cohomological equations. For minimal base transformations, the
well-known Gottschalk-Hedlund Theorem states that the existence of
continuous solutions is equivalent to the boundedness of the
associated cocycle.
\begin{thm}[Gottschalk-Hedlund] \label{t.gottschalk-hedlund}
  Suppose $\Omega$ is a compact metric space,
  $\Phi:\T\times\Omega\to\Omega,\ (t,x) \mapsto \Phi_t(x)$ is a
  minimal flow with time $\T=\Z$ or $\R$, $\rho\in\R$ and
  $\alpha:\T\times\Omega\to\R,\ (t,x)\mapsto\alpha_t(x)$ is a
  continuous cocycle over $\Phi$, that is,
  \begin{equation}
  \label{e.cocycle-equation}
  \alpha_{s+t}(x)\ =\ \alpha_s(\Phi_t(x))+\alpha_t(x) \ .
  \end{equation}
  Then the cohomological equation
  \begin{equation}
    \label{e.gh-ce}
    \alpha_t \ = \ \gamma - \gamma\circ \Phi_t + t\rho \quad \forall t\in\T
  \end{equation}
  has a continuous solution $\gamma : \Omega\to \R$ if and only if there
  exists $x\in\Omega$ such that $\sup_{t\in\T} \left| \alpha_t(x)
    -t\rho \right| < \infty$.  Moreover, in this case $\gamma$ can be
  defined by 
  \begin{equation}\label{e.coboundary}
  \gamma(x) \ = \ \limsup_{t\to\infty} \alpha_t(x)-t\rho \ .
\end{equation}
If $\Phi$ is not minimal, then (\ref{e.coboundary}) yields at least a bounded
solution of (\ref{e.cocycle-equation}).
\end{thm}
A closely related question is that of the existence of a semiconjugacy, also
called a factor map, from a given map to a rigid translation. For example,
suppose $f$ is an orientation-preserving homeomorphism of the circle with lift
$F:\R\to\R$, and $\rho=\nLim (F^n(x)-x)/n$ is the rotation number of
$f$. We say a continuous onto map $h:\kreis\to\kreis$ is a \emph{semi-conjugacy}
from $f$ to the rotation $R_\rho$ with angle $\rho$ if there holds $h\circ f =
R_\rho\circ h$. Writing $h(x) = x + \gamma(x) \bmod 1$ with continuous
$\gamma:\kreis \to \R$, it is easy to see that $h$ is a semi-conjugacy if
$\gamma$ is a solution of the cohomological equation (\ref{e.gh-ce}) with
$\Phi_n(x)=F^n(x)$ and $\alpha_n(x) = F^n(x)-x$. In comparison to the
Gottschalk-Hedlund Theorem, it is noteworthy that this equation always has a
continuous solution provided that $\rho$ is irrational, such that $R_\rho$ is
minimal, independent of any minimality assumption on $f$. This is the main
content of the celebrated Poincar\'e Classification Theorem.
\begin{thm}[Poincar\'e] \label{t.poincare}Suppose $f$ is an orientation-preserving
  homeomorphism of the circle with lift $F:\R\to\R$. Then the limit
  $\rho(F)=\nLim (F^n(x)-x)/n$ exists and is independent of $x\in\R$.
  Furthermore, there holds
  \alphlist
  \item[(i)] $\rho(F)$ is rational if and only if $f$ has a periodic orbit.
  \item[(ii)] $\rho(F)$ is irrational if and only if $f$ is
    semi-conjugate to the irrational rotation $R_{\rho(F)}$.
  \listend
\end{thm}
The core part of this result is the existence of a semiconjugacy to an
irrational rotation in (ii).  In general, we call the task of finding
equivalent conditions for the existence of a semiconjugacy to a given
minimal translation the {\em semiconjugacy problem}. (Here
`translation' will be understood in a broad sense, made precise
below.) Recently, the authors have independently made some progress on
this question. In \cite{jaeger/stark:2006}, a Poincar\'e-like
classification is given for skew-products on the two-torus over
irrational rotations (so-called quasiperiodically forced circle
homeomorphisms).  For maps on mathematical quasicrystals, a similar
classification was obtained in \cite{aliste:2010}. While in both
situations strong use is made of the order-preserving properties of
the considered systems, corresponding to the existence of a
one-dimensional foliation that is preserved by the dynamics, the
methods employed are quite different.

In order to obtain a unified proof of these results, we introduce a
general setting below that also allows to treat a number of other system
classes, including minimally forced circle homeomorphisms, increasing
maps on the real line with almost-periodic displacement (see
\cite{kwapisz:2000}) and scalar differential equations with
almost-periodicity both in space and time. All these different types
of systems have been studied rather independently in the literature so
far. The proposed framework should allow to study them from a more
global point of view and to identify similarities and differences between
them. To that end, we provide some structure results for our general
class of flows, which identify important subcases to which some of the
above examples can be associated.
\medskip

 \label{MainResults}

Suppose $\Omega$ is a compact metric space, $\varphi : \R \times \Omega \to
\Omega,\ (t,x) \mapsto \varphi_t(x) = x\odot t$ is a continuous flow on $\Omega$
and $\omega : \T \times \Omega\to \Omega,\ (t,x)\mapsto \omega_t(x)$ is a flow
with discrete or continuous time (that is, $\T=\Z$ or $\T=\R$) which commutes
with $\varphi$, that is,
\begin{equation}
  \label{e.omega_preserves_leaves}
  \omega_s(x\odot t) \ = \ \omega_s(x)\odot t \ .
\end{equation}
We say $\rho\in\R$ is {\em $(\varphi,\omega)$-irrational} if the {\em translation flow} 
\begin{equation}
  \label{e.translation}
  T^\rho : (t,x) \ \mapsto \  T^\rho_t(x) = \omega_t(x)\odot t\rho  
\end{equation}
is minimal. Otherwise we say that $\rho$ is {\em $(\varphi,\omega)$-rational}.  Now,
assume that $\tau :\T\times\Omega\to\R$ is a continuous function that
satisfies
\begin{equation}
  \label{e.tau-property}
  \tau_{s+t}(x) \ = \ \tau_s(\omega_t(x)\odot\tau_t(x)) + \tau_t(x) \ ,
\end{equation}
such that
\begin{equation}
  \label{e.fdefinition}
  \Phi = \Phi^{\varphi,\omega,\tau}  \ :\ \T \times \Omega \to \Omega \quad , \quad (t,x) \  
  \mapsto \ \Phi_t(x) = \omega_t(x)\odot \tau_t(x) \ 
\end{equation}
defines a flow with time $\T$ on $\Omega$. We call $\omega$ {\em
  transversal action} and $\tau$ {\em translation function} of
$\Phi$.  

An interpretation of the structure of (\ref{e.fdefinition}) is that there is a
one-dimensional foliation on $\Omega$ given by the orbits of $\varphi$, and the
flow $\Phi$ preserves this foliation in the sense that {\em leaves}
$\varphi_\R(x)=\{\varphi_t(x)\mid t\in\R\}$ are send to leaves
$\varphi_{\R}(\omega_t(x))$ by $\Phi_t$. We will further require $\Phi$
preserves the order on these leaves, which amounts to say that 
\begin{equation}
  \label{e.f_preserves_order}
  t + \tau(x\odot t) \ \geq \ \tau(x) \qquad \forall x\in\Omega, t \geq 0 \ .
\end{equation}
(See Remark~\ref{r.order} below.)  If the limit
\begin{equation}
  \label{e.rotation_number}
  \rho(\Phi) \ = \ \tLim \tau_t(x)/t \ . 
\end{equation}
exists and is independent of $x\in\Omega$, we call it the {\em
  translation number} of $\Phi$. Further, we say $\Phi$ has {\em
  bounded mean motion} if there exists a constant $C>0$ such that
\begin{equation}
  \label{e.bmm}
  \left| \tau_t(x)  - t\rho(\Phi) \right| \ \leq \
   C \qquad \forall x\in\Omega,\ t\in\T \ . 
\end{equation}

\begin{thm}
  \label{t.gotthead} Suppose $\rho(\Phi)$ is $(\varphi,\omega)$-irrational. Then the
  cohomological equation
  \begin{equation}
    \label{e.cohomological_equation}
    \tau_t(x) \ = \ \gamma(x) -\gamma\circ \Phi_t(x) + t\rho(\Phi) \ 
  \end{equation}
  has a continuous solution $\gamma:\Omega\to\R$ if and only if $\Phi$ has bounded
  mean motion. In this case, the map
  \begin{equation}
    \label{e.h_definition}
    h(x) \ = \ x\odot \gamma(x)
  \end{equation}
  provides a semi-conjugacy from $\Phi$ to the translation flow
  $T^{\rho(\Phi)}$, that is, $h\circ \Phi_t = T^{\rho(\Phi)}_t\circ h$
for all $t\in\T$. Moreover, $h$ can be chosen to be 
  order-preserving on the leaves $\varphi_\R(x), x\in\Omega$, that is,
  (\ref{e.f_preserves_order}) holds with $\tau=\gamma$. 
\end{thm}
A further immediate consequence we obtain from the proof in Section~\ref{Proof}
is the following.
\begin{cor} \label{c.minimalsets}
  Under the assertions of Theorem~\ref{t.gotthead} the flow $\Phi$ has
  a unique minimal set.
\end{cor}

The application of these results to the different examples mentioned above will
be discussed in detail in Section~\ref{Applications}. We close the introduction
with a remark on (\ref{e.f_preserves_order}) and the order-preservation on the
leaves.
\begin{rem} \label{r.order} We call $\varphi_\R(x)$ the {\em leaf} of $x$ and
  say the leaf is periodic if there exists $t_0>0$ with
  $\varphi_{t_0}(x)=x$. The natural order on $\R$, respectively $\R/t_0\Z$,
  induces a canonical order on $\varphi_\R(x)$. If the leaf is aperiodic,
  that is, $\varphi_t(x)\neq x$ for all $t\neq 0$, then this order is linear and
  given by $x\leq y :\equi y=x\odot t$ for some $t\geq 0$. When the leaf is
  periodic with period $t_0$, then the order is circular and is declared by
  $x\leq y \leq z$ being equivalent to the existence of $0\leq s\leq t \leq t_0$
  with $y=x\odot s$ and $z=x\odot t$ ($y$ lies between $x$ and $z$ in positive
  direction).  In both cases, (\ref{e.f_preserves_order}) implies that the flow
  $\Phi$ preserves the order on the leaves, that is, $x\leq y\leq z$ implies
  $\Phi_t(x)\leq \Phi_t(y) \leq \Phi_t(z) $ for all $t\in\T$.

  Given $y\in\varphi_\R(x)$, we write $[x,z]=\{y\in\varphi_\R(x) \mid x\leq
  y\leq z\}$ to denote intervals in the leaves with respect to the canonical
  ordering.
\end{rem}

\section{Proof of Theorem~\ref{t.gotthead}}
\label{Proof}
Due to (\ref{e.tau-property}) the function $\tau$ is a cocycle
over the flow $\Phi$, that is,
\begin{equation}
  \label{eq:4}
  \tau_{s+t}(x) \ = \ \tau_s(\Phi_t(x)) + \tau_t(x)\ . 
\end{equation}
By (\ref{e.bmm}) and the Gottschalk-Hedlund Theorem
\begin{equation}
  \label{e.gamma_def}
  \gamma(x) \ = \ \limsup_{t\to\infty}(\tau_t(x) -t\rho )\ 
\end{equation}
defines a bounded solution $\gamma :\Omega\to \R$ of the cohomological
equation (\ref{e.cohomological_equation}), and the restriction of
$\gamma$ to any minimal set $M$ is continuous. As a consequence,  the map
$h(x) \ = \ x\odot\gamma(x)$ satisfies
\begin{equation}
  \label{e.semi-conjugacy}
  h \circ \Phi_t \ = \ T^{\rho(\Phi)}_t\circ h \quad \forall t\in\T.
\end{equation}
 Moreover, (\ref{e.tau-property}) and the fact that $\Phi$ preserves the order 
on leaves imply that
 \begin{equation}
  \label{e.gamma_preserves_order}
  t+\gamma(x\odot t) \ \geq \ \gamma(x) \ ,
\end{equation}
which means that $h$ also preserves the order on leaves (see Remark~\ref{r.order}). It
remains to show that $\gamma$ is continuous on all of $\Omega$. 

In order to do so, fix any minimal set $M$ and let $\rho=\rho(\Phi)$.
Since $\gamma_{|M}$ is continuous, $h_{|M}$ is continuous as well. Due
to (\ref{e.semi-conjugacy}) the set $h(M)$ is $T^\rho$-invariant, such
that by $(\varphi,\omega)$-irrationality of $\rho$ we have $h(M)=\Omega$.  Given
$x\in\Omega$, we let ${\cal R}_M(x) = \{t\in\R \mid x\odot t\in M\}$.
Since $\gamma$ is bounded and $h$ maps $x\odot \R \cap M$ onto $x\odot
\R$, the set ${\cal R}_M(x)$ is relatively dense and the size of its maximal gap is
at most $2\sup_{x\in\Omega}|\gamma(x)|$. Given $x\in M$, define
\begin{eqnarray}
  \label{eq:1}
  t^+(x) \ = \ \inf\{t\geq 0 \mid x\odot t\in M\} & \quad   , \quad &  t^-(x) \ 
  = \ \sup\{t\leq 0 \mid x\odot t \in M\} \\
  x^+  \ = \ x\odot t^+(x) & \quad \mathrm{and}\quad  & x^- \ = \ x\odot t^-(x) \ .
\end{eqnarray}
Then $t^+$ and $t^-$ are bounded and due to the compactness of $M$ the
function $t^+$ is lower semi-continuous and $t^-$ is upper
semi-continuous. Furthermore, since $h$ preserves the order on the
leaves and maps $x\odot \R \cap M$ onto $x\odot \R$, it must be constant
on the interval $[x^-,x^+]$, that is,
\begin{equation}
\label{e.plateau}
h(x^-)=h(x^+)=h(x) = h(x') \qquad \forall x'\in[x^-,x^+] \ .
\end{equation}
In particular, this means that 
\begin{equation}
  \label{eq:3}
  \gamma(x) \ = \ \gamma(x^\pm) + t^\pm(x) \ .
\end{equation}
Fix $x\in \Omega$ and suppose $x'_n$ is a sequence in $\Omega$
converging to $x$. Let $x_n$ be a subsequence of $x_n'$ such that
$\gamma(x_n)$ converges. Let $t^+_n=t^+(x_n)$ and $t^-_n=t^-(x_n)$.
Taking a subsequence again if necessary, we may assume that $t^+_n\to
s^+$, $t^-_n\to s^-$, $x^+_n\to y^+$ and $x^-_n\to y^-$ as
$n\to\infty$. We have $y^+=x\odot s^+$, $y^-=x\odot s^-$ and, since
$x^\pm_n\in M$, the points $y^+$ and $y^-$ also belong to $M$. By
definition of $t^\pm$, we therefore have $s^+\geq t^+$ and $s^-\leq
t^-$. Further
\begin{equation}
  \label{eq:2}
  \nLim \gamma(x_n) \ \stackrel{(\ref{eq:3})}{=} \ \nLim \gamma(x^+_n) +
  t^+_n \ = \ \gamma(y^+) + s^+ \ = \
  \gamma(x\odot s^+) + s^+ \ \geq \ \gamma(x) \ . 
\end{equation}
In the same way we obtain $\nLim \gamma(x_n) = \gamma(y^-)+s^- \leq
\gamma(x)$, which implies $\nLim\gamma(x_n) = \gamma(x)$. This
argument shows that any convergent subsequence of $\gamma(x'_n)$ has
limit $\gamma(x)$, and since $\gamma$ is bounded we conclude $\nLim
\gamma(x_n')=\gamma(x)$.  Since $x$ and $x_n'$ were arbitrary, this
shows the continuity of $\gamma$. \qed\medskip
 
\proof[\bf Proof of Corollary~\ref{c.minimalsets}.] Suppose for a contradiction
that $M_1$ and $M_2$ are two different minimal sets of $\Phi$. Then $M_1\cap
M_2=\emptyset$ by minimality, such that $M_2\ssq M_1^c$. As $M_1$ is compact,
the set $\varphi_\R(x)\cap M_1^c$ consists of an at most countable union of open
segments in $\varphi_\R(x)$, each of which is mapped to a single point by $h$
(see~(\ref{e.plateau})). However, this means that $h(\varphi_\R(x)\cap M_2)$ is
at most countable, contradicting the fact that $h(\varphi_\R(x)\cap
M_2)=\varphi_\R(x)$ since $h(M_2)=\Omega$. \qed

\section{Two structure results}
\label{StructureResults}

In this section, we study the structure of the joint action of $\varphi$ and
$\omega$ and the implications for $\Phi$. In particular, we concentrate on some
situations where the general setting introduced above reduces to a simpler
one. As a standing assumption, we assume the minimality of the joint action of
$\varphi$ and $\omega$. Note that otherwise no $(\varphi,\omega)$-irrational numbers can
exist.

Our aim is to show that if a periodic leaf exists, then $\Phi$ `almost'
reduces to a skew product circle flow.
\begin{prop}\label{p.periodic}
  Suppose that the joint action of $\varphi$ and $\omega$ is minimal, and that
  $\varphi$ has a periodic leaf. Then there exists an $\omega$-minimal set $W$
  and a skew product flow $\Psi$ on $W\times\kreis$ which is a finite-to-one
  extension of the flow $\Phi$ defined in (\ref{e.fdefinition}). 

  Further, there exists a continuous lift $\widehat\Psi:W\times\R\to W\times\R$
  of $\Psi$ whose displacement function
  $\tilde\tau_t(x,s)=\pi_2(\widehat\Psi_t(x,s))-s$ is a scalar multiple of
  $\tau\circ g \circ p$, where $g:W\times\T^1\to \Omega$ is the finite-to-one
  factor map and $p:W\times\R\to W\times\T^1$ is the canonical projection.
\end{prop}

Note that the second part of the statement implies that for all questions
related to the translation number, one may equivalently consider $\Psi$ instead of $\Phi$.
\smallskip

In order to prove the above statement, we start with some general
considerations. First, given any closed set $W$, we define its invariant group
$H_W$ as
\[H_W :=\{ \rho \in \R \mid W \odot \rho = W \}.\] From the continuity of
$\varphi$, it is clear that $H_W$ is a closed subgroup of $\R$. This means that
$H_W$ is either trivial, $\R$, or a lattice in $\R$.
\begin{lem}
  \label{p.H_W}
Suppose that the joint action of $\varphi$ and $\omega$ is minimal. Then there
exists a closed subgroup $H = H_{\varphi,\omega}$ of $\R$ 
such that $H_W = H$ for all (non-empty) $\omega$-minimal sets $W$. 
\end{lem}
\begin{proof}
  Fix a (non-empty) $\omega$-minimal set $W$ and set $H = H_W$. We deal with the three
  possible cases for $H$ separately.

First, suppose that $H = \R$. Then $W$ is invariant for the joint action. Since this 
action is minimal, it follows that $W = X$ and therefore there is only one
non-empty $\omega$-minimal set, from which the conclusion is direct.

Second, suppose that $H$ is a lattice, that is, $H = s_0\Z$ for some $s_0>0$. On
one hand, due to the continuity of $\varphi$ we have that 
\[Y =\cup_{t\in[0,s_0)}W \odot t\] 
is closed. Since it is also invariant for the
joint action, it follows that $Y = X$. On the other hand, it is easy to see that
$W \odot t$ is an $\omega$-minimal set for every $t\in\R$. This implies that
$t-s\in s_0\Z$ whenever $W \odot t \cap W \odot s$ is not empty. It follows that
$\{W \odot t\,:\,t\in[0,s_0)\}$ is a partition of $X$ by $\omega$-minimal sets.
In particular, $H_W=s_0\Z$ for all $\omega$-minimal sets $W$. 

Finally, if $H_W = \{0\}$, then for every $\omega$-minimal set $W'$ we have
$H_{W'}=\{0\}$. Indeed, if we suppose otherwise, then by the above
$H_W$ is either $\R$ or a lattice, which is absurd.
\end{proof}

We now focus on the case where $H$ is a lattice, that is, $H=s_0\Z$. In this
case, it turns out that $\varphi$ is conjugate to a suspension flow, which we
define as follows. Let $\sigma:W\to W,\ x\mapsto x\odot s_0$. Note that $\sigma$
is continuous and commutes with $w|_W$. Let $\sim_\sigma$ the smallest
equivalence relation on $W\times\R$ that satisfies
$(x,s+ks_0)\sim_\sigma(\sigma^k(x),s) \ \forall (x,s)\in W\times\R,\ k\in\Z$. We
denote by $[x,s]_\sigma$ the equivalence class of $(x,s)\in W\times\R$ with
respect to $\sim_\sigma$ and let $W_\sigma=W\times\R/\sim_\sigma$ be the
{\em suspension space}. Then we define the {\em suspension flow} $\varphi^{W,\sigma}$ on
$W_\sigma$ by $\varphi^{W,\sigma}_t[x,s]_\sigma = [x,s+t]_\sigma$.

Now, it is easy to check that the map $h:W_\sigma\to\Omega,\ h([x,s]_\sigma) =
x\odot s$ is a homeomorphism. Further, we have $h\circ
\varphi^{W,\sigma}_t([x,s]_\sigma) = x\odot (s+ t) = h([x,s]_\sigma)\odot t =
\varphi_t\circ h([x,s]_\sigma)$, such that $h$ conjugates $\varphi^{W,\sigma}$
and $\varphi$. We have thus obtained the following statement.

\begin{cor} \label{c.struc_space} If $H$ is a lattice then, for every
  $\omega$-minimal set $W$, the flow $\varphi$ is conjugate to the suspension
  flow $\varphi^{W,\sigma}$ defined above. 
\end{cor}

As a consequence, we can also view the flow $\Phi$ as a flow on the suspension
space $W_\sigma$.  Namely, let $\tilde\Phi : W_\sigma\times\R \to W_\sigma,\
\tilde\Phi_t([x,s]_\sigma) = [\omega(x),s+\tau_t(x\odot s)]_\sigma$. Then it is
easy to check that $h$ defined above conjugates $\tilde\Phi$ and $\Phi$. 
\begin{cor}
\label{c.struct_map}
If $H$ is a lattice, then the flow $\Phi$ from (\ref{e.fdefinition}) is
conjugate to the flow $\tilde \Phi$ on the suspension space $W_\sigma$.
\end{cor}

We now turn to the situation where $\varphi$ has a periodic leaf. In this case,
all leaves are periodic.
\begin{lem}
\label{l.periodic_orbit}
Suppose that the joint action of $\varphi$ and $\omega$ is minimal and $\varphi$
has a periodic leaf. Then all the leaves of $\varphi$ are periodic with the
same period.
\end{lem}
\begin{proof}
  Suppose $t$ is the period of a leaf $\varphi_\R(x)$, such that $x\odot t =
  x$. Since $\omega$ commutes with $\varphi$, it is clear that $
  \omega_s(x)\odot t = \omega_s (x)$. Hence, $\omega$ sends periodic leaves to
  periodic leaves. Fix any $y\in X$. Due to the minimality of the joint action
  of $\varphi$ and $\omega$ there exist sequences $(t_k)_{k\in\N}$ and
  $(s_k)_{k\in\N}$ in $\T$ such that $y_k=\omega_{t_k}(x)\odot s_k$ converges to
  $y$. Since all $y_k$ are $t$-periodic, it follows that $y\odot t =
  \lim_{k\to\infty}y_k\odot t = \lim_{k\to\infty}y_k = y$. From this, we easily
  obtain that all leaves have the same period.
\end{proof}

\begin{proof}[Proof of Proposition~\ref{p.periodic}]
  Fix an $\omega$-minimal set $W$ and let $s_0,\ \sigma$ and $W_\sigma$ be as
  above. By Lemma \ref{l.periodic_orbit}, all the orbits of $\varphi$ are
  periodic with period $t_0$. In particular we have $W \odot t_0 = W$, such that
  $H_W$ is a lattice $s_0\Z$ and $t_0 = k s_0$ for some $k\in\Z$. Consequently,
  $\sigma^k$ is the identity map on $W$. 

  Now we define a flow $\widehat \Psi$ on $W\times\R$ by $\widehat\Psi_t(x,s) =
  (\omega_t(x),s+t_0^{-1}\tau_t(x\odot t_0s))$. Further, we let $\hat g :
  W\times\R\to\Omega$ be given by $g(x,s) = x\odot t_0s$. Then
  \begin{eqnarray*}
    \hat g \circ \widehat \Psi_t(x,s) & = & \omega_t(x)\odot (t_0s+\tau_t(x\odot t_0s)) \\
    &=& \omega_t(x\odot t_0s) \odot \tau_t(x\odot t_0s) \\ & = & \Phi_t(x\odot t_0s)
    \ = \ \Phi_t\circ \hat g(x,s) \ . 
  \end{eqnarray*}
  Hence, $\hat g$ is a semiconjugacy from $\widehat\Psi$ to $\Phi$. Furthermore,
  $\widehat \Psi$ projects to a skew product circle flow $\Psi$ on
  $W\times\kreis$ and $\hat g$ projects to a map $g:W\times\kreis\to\Omega$. It
  is easy to check that this map $g$ is a finite-to-one factor map (or
  semiconjugacy) from $\Psi$ to $\Phi$. The fact that the displacement function
  of the lift $\hat\Psi$ has the required form follows immediately from the
  above definition.
\end{proof}
A second situation in which the general setting simplifies is when there exists
an invariant leaf.
\begin{prop} \label{p.trivial-omega}
  Suppose that the joint action of $\varphi$ and $\omega$ is minimal
  and that there is a $\varphi$-orbit that is invariant by $\omega$.
  Then there exists $\rho\in\R$ such that $\omega_t(x) = x \odot t\rho$
  for all $x\in X$ and $t\in\T$.
\end{prop}
\begin{proof}
  Suppose $\varphi_\R(x)$ is $\omega$-invariant. Given $t\in\T$ there
  exist $\rho(t)\in\R$ such that $\omega_t(x) = x \odot t\rho(t)$. Due the minimality
  of the joint action, for every $y\in X$ we can find sequences
  $(t_k)_{k\in\N}\subset\R$ and $(s_k)_{k\in\N}\subset\Z$ such that
  $\lim_{k\rightarrow\infty} y_k = \lim_{k\rightarrow\infty} \omega_{t_k}(x)\odot s_k = y$. It
  follows from the continuity of $\omega$ that $\omega_t(y_k)$ converges to
  $\omega_t(y)$. Since the actions of $\omega$ and $\varphi$ commute, we deduce
  that
  \begin{eqnarray*}
    \omega_t(y) & = & \lim_{k\to\infty}\omega_t(y_k) \\ & = & \lim_{k\to\infty}
    \omega_{t_k}(\omega_t(x))\odot s_k \ =  \lim_{k\to\infty}(\omega_{t_k}(x)\odot
     t\rho(t))\odot s_k \\ & = & \lim_{k\to\infty} \left(\omega_{t_k}(x)\odot s_k\right)
     \odot t\rho(t) \ = \ y\odot t\rho(t) \ . 
  \end{eqnarray*}
  The above equality holds for all $y\in \Omega$, such that for all $t\in\T$
  there exists a unique $\rho(t)$ with $\omega_t(y)=y\odot t\rho(t) \ \forall
  y\in\Omega$. A simple compatibility check yields that $\rho(t)$ is the same
  for all rational (or integer) $t\in\T$. If $\T=\Z$ this completes the proof,
  otherwise the statement follows from the continuity of $\omega$.
\end{proof}

\section{Applications} \label{Applications}

Before we turn to more sophisticated applications, let us indicate how the two
classical examples mentioned in the introduction fit into this framework. For
the Gottschalk-Hedlund Theorem, we choose $\T=\Z$ and $\varphi_t=\Id_\Omega \
\forall t\in\T$.  Thus $\Phi=\omega$, and in particular $\Phi$ is independent of
$\tau$. This implies that every $\rho\in\R$ is $(\varphi,\omega)$-irrational provided that
$\omega$ is minimal. Hence, in this particular case, we see that Theorem
\ref{t.gotthead} is equivalent to Gottschalk-Hedlund Theorem.

In order to obtain part (ii) of the Poincar\'e Classification Theorem we let
$\Omega=\kreis$, $\varphi_t(x)=x+t\bmod 1$ and choose $\omega$ to be trivial,
that is, $\omega_t=\Id_\Omega \ \forall t\in\T=\Z$. Further, given $t\in\Z^2$ we
choose $\tau_t(x)=F^t(x)-x$, where $F:\R\to\R$ is a lift of the
orientation-preserving circle homeomorphism $f$. Then $\Phi$ and $T^{\rho}$ with
$\rho=\rho(\Phi)=\rho(F)$ are given by their time-one maps $\Phi_1=f$ and
$T^\rho_1 = R_\rho$. Further, $|F^t(x)-x-t\rho|\leq 1 \ \forall x\in\R,\ t\in\Z$
(see, for example, \cite{katok/hasselblatt:1997}). Hence, when $\rho$ is
irrational, such that $T^\rho$ is minimal, then Theorem~\ref{t.gotthead}
provides the required semiconjugacy from $f$ to $R_\rho$. \medskip

We now want to apply Theorem~\ref{t.gotthead} to some important
classes of examples. The skew products in Section~\ref{SkewProducts}
correspond to the case of periodic leafs described in
Proposition~\ref{p.periodic}. The examples from
Section~\ref{Aperiodic} all have in common that the transversal flow
$\omega$ is trivial, such that leafs are preserved one-by-one as in
Proposition~\ref{p.trivial-omega}. Only in Section~\ref{APspacetime}
we then make full use of the generality of (\ref{e.fdefinition}), by
considering systems that have both non-trivial transversal action and
non-periodic leafs.

\noindent 
\subsection{Skew products circle flows over a minimal base}\label{SkewProducts}
Suppose $\Omega_0$ is a compact metric space and
$\tilde\omega:\T\times\Omega_0\to\Omega_0$ a minimal flow with time $\T=\Z$ or
$\R$. Let $\Omega=\Omega_0\times\kreis$ and write $x=(x_0,x')\in\Omega$. We
consider skew product flows
\begin{equation}\label{e.skew-product}
\Phi:\T\times \Omega\to\Omega
\quad , \quad \Phi_t(x_0,\bar x) \ = \ (\tilde\omega_t(x_0),\Phi^{x_0}_t(\bar x))\ 
\end{equation}
 over a {\em base flow} $\omega$. We say $\widehat\Phi$ is a lift of $\Phi$ if it
is a continuous map
\begin{equation}\label{e.skew-product-lift}\widehat\Phi:
 \T\times\Omega_0\times\R\to\Omega_0\times\R\quad ,\quad
(t,x_0,x')\mapsto (\tilde\omega_t(x_0),\widehat\Phi^{x_0}_t(x'))\ , 
\end{equation}
that satisfies $\widehat\Phi^{x_0}_t(x') \bmod 1 = \Phi^{x_0}_t(x'\bmod 1)$. 
\begin{cor}\label{c.skew-products}
  Suppose $\Phi$ is a skew product flow of the form (\ref{e.skew-product}) with
  lift $\widehat\Phi$ and the product flow $T^\rho : (t,x_0,x') \mapsto
  (\tilde\omega_t(x_0),x'+t\rho\bmod 1)$ is minimal. Then $\Phi$ is semiconjugate to
  $T^\rho$ if and only if
  \begin{equation}\label{e.bmm-skewproducts}
     \sup_{t\in\T, x\in\Omega} \left|\widehat\Phi_t^{x_0}(x')-x'-t\rho\right| \
     \leq \ \infty \ .
   \end{equation}
 \end{cor}
 \proof It suffices to note that $\Phi$ is of the form (\ref{e.fdefinition})
 with $\varphi_t(x_0,x') = (x_0,x'+t\bmod 1)$,
 $\omega_t(x_0,x')=(\tilde\omega_t(x_0),x')$ and $\tau_t(x_0,x') =
 \widehat\Phi_t^{x_0}(x')-x'$.  \qed\medskip

 We note that so far the result was only known in the case where $\omega_0$ is
 an irrational rotation of the circle \cite{jaeger/stark:2006}. The proof given
 in \cite{jaeger/stark:2006} can be extended to irrational rotations of
 higher-dimensional tori, but does not carry over to the general situation
 described here.\medskip

Corollary~\ref{c.minimalsets} further implies that under the hypotheses of
Corollary~\ref{c.skew-products} there is a unique minimal sets. However, in this
case a result by Huang and Yi yields the uniqueness of the minimal set also when
mean motion is unbounded, that is, when (\ref{e.bmm-skewproducts})
fails \cite[Theorem 7]{huang/yi:2007}. This leads to the following
\begin{cor}
  Suppose $\Phi$ is a skew product flow of the form (\ref{e.skew-product}) and
  the product flow $T^\rho : (t,x_0,x') \mapsto
  (\tilde\omega_t(x_0),x'+t\rho\bmod 1)$ is minimal.  Then $\Phi$ has a unique
  minimal set.
\end{cor}

Finally, we note that the uniqueness of the rotation number of skew product
flows depends on the unique ergodicity of the base flow.
\begin{thm}[\cite{herman:1983}] \label{t.herman}
  Suppose $\Phi$ is a skew product circle flow of the form
  (\ref{e.skew-product}) and the base flow $\omega$ is uniquely ergodic. Then
  the limit
  \[
    \rho(\hat\Phi) \ = \ \lim_{t\to\infty} \Phi_t^{x_0}(x')/t \ 
  \]
  exists and does not depend on $x_0\in\Omega_0$ or $x'\in\R$.
\end{thm}

\subsection{Aperiodic order on the real line} \label{Aperiodic}
Let $f:\R\rightarrow\R$ be a continuous non-decreasing map. The translation
number of $f$ at $x$ is defined by
\[\rho(f,t):=\lim_{n\rightarrow\infty} \frac{f^{n}(t)-t}{n}\]
provided the limit exists. As pointed out in the introduction, in the case that $f$ is the
lift of an orientation-preserving homeomorphism of the circle, the translation
number of $f$ at $x$ exists for every $x\in\R$ and does not depend on
$x$. Notice that in this case the displacement function, which is defined as
$\phi:=f-id$ is $1$-periodic. In this section, we are interested in the case
when $\phi$ is no longer periodic but has some sort of ``quasiperiodicity''.
a
\subsubsection{Almost-periodic displacements}
We first deal with the case where $\phi$ is almost periodic.  Recall that a function $\phi$ is almost
periodic (in the sense of Bohr), if the family $\{\phi(\cdot + \tau) : \tau \in
\R\}$ has compact closure in $C(\R)$ endowed with the topology of uniform
convergence. We write $X_\phi$ for the closure and $\xi_t = \phi(\cdot + t)$ for
all $t\in\R$. There is a natural multiplication defined on $X_\phi$. Namely if
$\xi = \lim_{n} \xi_{t_n}$ and $\xi' = \lim_{n} \xi_{s_n}$ belong to $X_\phi$,
then
\[\xi \cdot \xi':=\lim_{n} \xi_{t_n+s_n}\]
is well-defined and $(X_\phi,\cdot)$ is a compact Abelian group (see, for
instance, \cite{ellis:1969} for details) with identity $\xi_0$. Note that
$\xi\cdot\xi_t(\ldot)=\xi(\ldot+t)$, and this defines a natural flow
$\phi:(t,\xi)\mapsto \xi\cdot\xi_t=:\xi\odot t$ on $X_\phi$. 

The map $f:\R\rightarrow\R$ given by $f(x)=x+\phi(x)$ can now be studied as
follows. We let $\tau:X_\phi\to\R,\ \xi \mapsto\xi(0)$ and define
$F:X_\phi\rightarrow X_\phi$ by
\begin{equation}\label{e.ap-hullmap}
F(\xi)\ =\ \xi \odot \tau(\xi) \ .
\end{equation}
Then
$F(\xi_t)=\xi_t\odot\xi_t(0)=\xi_t\odot\phi(t)=\xi_{f(t)}$. Consequently
$h:\R\to X_\phi,\ t\mapsto\xi_t$ satisfies $h\circ
f(t)=\xi_{f(t)}=F(\xi_t)=F\circ h(t)$, such that $f$ can be identified with
$F_{|\{\xi_t\mid t\in\R\}}$. In this setting, uniqueness of the rotation number
was shown by Kwapisz.
\begin{thm}[\cite{kwapisz:2000}]
\label{thm:KwapiszUnique}
If $f$ is orientation preserving and $\phi$ is bounded away from zero, then $F$
has a unique translation number $$\rho(F)\ = \ \nLim \ntel\inergsum \tau\circ
F^i(\xi) \ = \ \nLim \ntel\inergsum \tau\circ f^i(t)\ ,$$ meaning that the
limits exist and are independent of $\xi\in X_\phi$ and $t\in\R$.
\end{thm}
Note that the system obtained in (\ref{e.ap-hullmap}) can be written in the
general form (\ref{e.fdefinition}) with time $\T=\Z$ and trivial transversal
action $\omega=\Id$. It follows that Theorem \ref{t.gottschalk-hedlund} can be
applied to give a solution to the semiconjugacy problem for maps
of this type,
which was explicitly raised in \cite{kwapisz:2000}.  However, before stating the
result, we first want to understand which numbers are
$(\varphi,\omega)$-rational. 

It is well-known that every continuous almost periodic function $q$ has a mean
value
\[ M(q) = \lim_{N\rightarrow+\infty} \frac{1}{N}\int_{-N}^{N} q(t)dt.\] The
\emph{frequency module} of $\phi$ is defined to be the set
\[\mathcal{M} = \left\{\sum_{v} j_v \lambda_v, j_v\in\Z\right\},\]
where the $\lambda_v\in\R$ are the (countably many) values of $\lambda$ for
which $M(\phi e^{-i\lambda x})\neq 0$. We say $t\in\R$ is {\em rationally
  independent of } $\mathcal{M}$ if $kt\in\mathcal{M}$ with $k\in\Z$
implies $k=0$. Further, we call $\lambda\in\R$ a {\em continuous eigenvalue} of
the flow $(X_\phi,\R)$ if there exists a continuous map $g: X_\phi\to\R$ which
satisfies $g(\xi\odot t)=e^{i\lambda t}g(\xi)$ for all $\xi\in X_\phi$ and
$t\in\R$.
\begin{prop}[\cite{miller:2004}]
\label{prop.Miller}
The frequency module of $f$ coincides with the group of continuous eigenvalues
of $(X_\phi,\R)$.
\end{prop}
\begin{lem}[{\cite[4.24.1]{Glasner}}]
\label{lem.Glasner}
Let $(X,\{T_t\}_{t\in\R})$ be a minimal compact flow. Then for every
$t\in\R\setminus\{0\}$, the system $(X,T_t)$ is not minimal if and only if there
is a continuous eigenvalue $\lambda$ of $(X,\{T_t\}_{t\in\R})$ such that
$t\lambda$ is an integer.
\end{lem}
This allows to characterise the $(\varphi,\omega)$-irrational times. 
\begin{prop}Suppose the flow $\Phi$ in (\ref{e.fdefinition}) is of the
  form given by (\ref{e.ap-hullmap}). Then $\rho$ is
  $(\varphi,\omega)$-irrational if and only if
 $\rho^{-1}$ is rationally independent of $\mathcal{M}$, .
\end{prop}
\proof By definition, $\rho$
is $(\varphi,\omega)$-irrational iff the multiplication by $\xi_{\rho}$ is
minimal. By Lemma \ref{lem.Glasner}, this happens iff for every continuous
eigenvalue $\lambda$ the quantity $\lambda\rho$ is not an integer. By
Proposition \ref{prop.Miller}, this is equivalent to the fact that $k\rho^{-1}$
is not a frequency in $\mathcal{M}$ for all $k\in\Z\setminus\{0\}$, which means
that $\rho^{-1}$ is rationally independent of $\mathcal{M}$. 
\endproof
\begin{cor}
\label{thm:conjKwapisz}
Suppose that the hypotheses of Theorem \ref{thm:KwapiszUnique} are
satisfied. Further, assume that $\rho(F)^{-1}$ is rationally independent of
$\mathcal{M}$ and that
\begin{equation}
\sup_{n,x}|F^{n}(x)- x - n\rho(x)|<\infty.
\end{equation} 
Then $F$ is semiconjugate to the multiplication by $\xi_{\rho(F)}$. 
\end{cor}

\subsubsection{Pattern-equivariant displacements with respect to quasicrystals}

We now show how Theorem \ref{t.gottschalk-hedlund} can be applied to the case
where the displacement $\phi$ is pattern-equivariant with respect to a
mathematical quasicrystal of the real line. We keep the exposition brief, for
more details see \cite{aliste:2010}. Recall that a Delone set in $\R$ is an
increasing sequence $X=(\theta_n)_{n\in\Z}$ of points in $\R$ such that
$0<\inf_{n}|\theta_n-\theta_{n-1}|$ and $\sup_{n}|\theta_n-
\theta_{n-1}|<+\infty$. If the set $\{\theta_n-\theta_{n-1}\}_{n\in\Z}$ is
finite, we say that $X$ has \emph{finite local complexity}. Given $x\in X$ and
$R>0$, the $R$-patch of $X$ centred at $x$ is defined as $P(x,R) = X\cap
[x-R,x+R]$. Two patches $P(x,R)$ and $P(y,R)$ are equivalent if $P(x,R)- x =
P(y,R) - y$.  We say that $X$ is \emph{repetitive} if for every patch
$P=P(x,R)$, there exists $M>0$ such that every interval of length $M$ contains
the centre of a patch equivalent to $P$. Given a patch $P(x,R)$, $t\in\R$ and
$N>0$, let
\[n(P,N,t) := \#\{y\in [t-N,t+N] \mid y\in X \text{ and } P(y,R) \sim P(x,R)\}\]
denote the number of patches of $X$ that are equivalent to $P$ and whose centre
is included in the interval $[t-N,t+N]$.  We say that $X$ has \emph{uniform
  patch frequencies} if $n(P,N,t)/2N$ converges uniformly in $t$ to a limit
$\nu(P)$, which is independent of $t$, when $N\rightarrow+\infty$. Finally, we
say that $X$ is a \emph{quasicrystal} if $X$ is non-periodic, repetitive and has
uniform patch frequencies. A function $\phi:\R\rightarrow\R$ is \emph{strongly
  $X$-equivariant} if there exists $R>0$ such that $ X - x \cap [-R,R] = X - y
\cap [-R,R] $ implies that \[ \phi(x) = \phi(y).\] A continuous function
$\phi:\R\rightarrow\R$ is \emph{$X$-equivariant} if it is the uniform limit of a
sequence of strongly $X$-equivariant continuous functions. A continuous function
$\phi$ is \emph{pattern-equivariant} if there exists a quasicrystal $X$ such
that $\phi$ is $X$-equivariant. The most well-known example for a quasicrystal
of the real line is given by the Fibonacci tiling, which can be constructed by
iterating the substitution
\begin{eqnarray*}
L \rightarrow LS,\\
S \rightarrow L
\end{eqnarray*}
starting with the sequence $L.L$ as depicted in Fig. \ref{r.fib} and then
replacing the infinite sequence thus obtained by a tiling in which each $L$ is
replaced by an interval of length $\tau$, where $\tau$ is the golden mean
number, and each $S$ by an unit-interval.

\begin{figure}
\begin{center}
\includegraphics{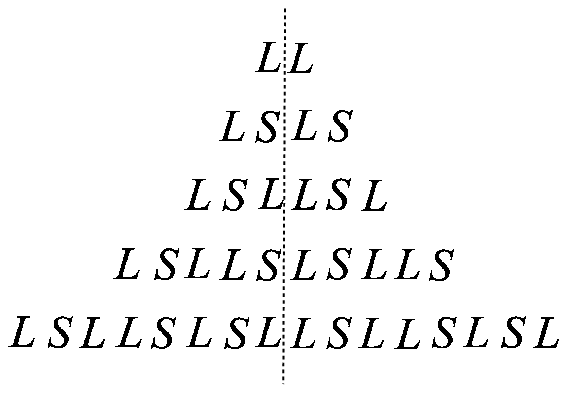}

\caption{The Fibonacci tiling}
\label{r.fib}
\end{center}
\end{figure}

Suppose that the displacement $\phi$ is $X$-equivariant for some quasicrystal
$X$. Then, similar to almost-periodic functions, there exists a compact metric
space $\Omega$, called the hull of $X$, which is the completion of the set
\[\{ X - t\mid t\in \R\}\] endowed with an appropriate metric (we omit the
precise definition of this metric here, see \cite{aliste:2010} for details). The
elements of $\Omega$ are all Delone sets, and there is a continuous action
$\Gamma$ on $\Omega$ given by translations $\Gamma_t Y = Y - t$ for all
$t\in\R$ and $Y\in\Omega$. The repetitivity of $X$ is equivalent to
the minimality of $\Gamma$ and the uniform pattern frequencies property is
equivalent to the unique ergodicity of $\Gamma$ (see \cite{aliste:2010} and
references therein for details). Moreover, there is a unique continuous function
$\tau:\Omega\rightarrow\R$ such that $\tau(X-t) = \phi(t)$ for all $t\in\R$, and
thus a continuous map $F:\Omega\rightarrow\Omega$ defined by
\begin{equation}
\label{eq.Delone} 
F(Y) = Y - \tau(Y)\quad, \quad Y\in\Omega. 
\end{equation}
The map $F$ is related to $f$ by the relation $F(X-t) = X - f(t)$ for all $t\in\R$. 
\begin{thm}[\cite{aliste:2010}]
  Suppose that $F:\Omega\rightarrow\Omega$ has the form of \eqref{eq.Delone}. If
  $\Phi$ does not have zeros, then $F$ has a unique translation number
  $\rho(F)$, defined as
  \[\rho(F):=\lim_{n\rightarrow\infty}\frac{1}{n}\sum_{k=0}^{n-1}\tau(F^{k}(X)).\]
\end{thm}
The following result (which already appears in \cite{aliste:2010}) is again a
direct consequence of Theorem~\ref{t.gottschalk-hedlund}.
\begin{cor}\label{alisteThesis}
  Suppose that $\Gamma_{\rho(F)}$ is minimal, and that
  $\sup_{n}|\sum_{k=0}^{n-1}\tau(F^{k}(X))-n\rho(F)|<\infty$. Then $F$ is
  semiconjugate to $\Gamma_{\rho(F)}$.
\end{cor}
\subsubsection{Special flows}
Let us end this section by noticing that both of the above hull constructions,
for functions with almost periodic and pattern-equivariant displacements, lead
to examples of special flows. Recall the definition of a special flow. Let
$\tilde{X}$ be a compact metric space, $\sigma:X\rightarrow X$ a minimal
transformation and $h:\tilde{X}\rightarrow \R^+$ a continuous function. Let
$X = \tilde{X}\times \R/\sim$, where $\sim$ is the smallest equivalence relation
such that $(\tilde{x},h(\tilde{x}))\sim (\tilde{\sigma}({x}),0)$ for all
$\tilde{x}\in\tilde{X}$. It is well-known that the natural flow over $\tilde
X\times \R/\sim$ descends to a well-defined flow over $X$, which satisfies
$\varphi_t [(\tilde{x},s)] = [(\tilde{x},s+t)]$ for all $\tilde{x}\in X$ and
$t,s\in\R$, where $[(x,s)]$ denotes the equivalence class of $(x,t)$. The
dynamical system $(X,\varphi)$ is known as the special flow over $\sigma$ with
roof function $h$.

In the case of almost periodic functions, the map $\sigma$ can be chosen as the
first return map to $X_{\phi,s}$ for $s>0$, where $X_{\phi,s}$ is the closure of
$\{\xi_{ns}\mid n\in\Z\}$ in the uniform topology. In the case of quasicrystals,
$\sigma$ is given by the first return map to the canonical transversal
$\Omega_0=\{\Omega\in X\mid 0\in\Omega\}$.

\subsection{Scalar differential equations almost periodic in time and space} \label{APspacetime}
The study of differential equations with almost periodic right side was already
initiated by Favard \cite{favard:1928} in 1928 and has a long tradition
since. \cite{shen/yi:1998} gives a good overview and more recent references. We
first recall some basic facts. Given $f:\R\times\R^d,\ (t,x) \mapsto f(t,x)$ we
define $f_t : \R \times \R^d \to \R$ by $f_t(s,x) = f(s+t,x)$. Then, similar to
above, $f$ is called {\em uniformly almost periodic} in $t$ if $\{ f_t\mid
t\in\R\}$ has compact closure in ${\cal C}^0(\R\times\R^d,\R^d)$ in the topology
of uniform convergence. In this case
\[
{\cal H}(f) \ := \ \overline{\{f_t\mid t\in\R\}} 
\]
is called the {\em hull} of $f$ and the flow 
\[
\omega\ :\ \R\times{\cal H}(f) \to {\cal H}(f)\quad , \quad (t,g) \mapsto
(t,g_t)
\]
is uniformly almost periodic \cite{ellis:1969} and minimal.  An almost periodic
scalar differential equation is of the form
\begin{equation} \label{e.ap-scalarODE}
x'\ = \ f(t,x)
\end{equation}
with $f:\R\times\R \to \R$ uniformly almost periodic in $t$. It  gives
rise to a skew product flow of the form
\[
\widehat \Phi \ : \ \R \times{\cal H}(f) \times \R \to {\cal H}(f)\times\R\quad , \quad
(t,g,x)\mapsto (g_t,\xi(t,g,x)) \ ,
\]
where $t\mapsto \xi(t,g,x)$ is the solution of the scalar differential equation
$x'=g(t,x)$ with initial values $t_0=0$ and $x_0=x$. (Here, we make the usual
assumptions of Lipschitz-continuity in $x$ in order to guarantee the existence
and uniqueness of solutions. It suffices to make these assumptions on $f$, since
they carry over to the hull \cite[Theorem 3.1 in Part 1]{shen/yi:1998}.

When no further structure concerning the dependence of $f$ on $x$ is assumed,
the study of (\ref{e.ap-scalarODE}) is mostly restricted to the description of
bounded solutions and their orbit closures \cite{johnson:1981}. Much more can be
said when $f$ is periodic in $x$ and therefore projects to a function on
$\R\times \T^1$. In this case $\widehat\Phi$ projects to a skew product circle
flow $\Phi$ over the minimal base $({\cal H}(f),\omega)$, and an extensive
theory exists to describe the dynamical behaviour \cite{huang/yi:2007}. As
mentioned above, in this situation the rotation number is unique by
Theorem~\ref{t.herman} and Corollary~\ref{c.skew-products} provides an
equivalent condition for the existence of a semiconjugacy to the translation
flow $(t,g,x) \mapsto (g_t,x+t\rho \bmod 1)$. \smallskip

However, the main interest of the general setting introduced in
Section~\ref{MainResults} in this context is the fact it also applies when $f$
is only almost periodic in $x$. Given $f:\R^2\to\R$ we define
$f_{t,x}:\R^2\to\R$ by $f_{t,x}(s,y)=f(s+t,y+x)$ and say that $f$ is {\em
  uniformly almost periodic in $t$ and $x$} if $\{f_{t,x} \mid (t,x)\in\R^2\}$
has compact closure in ${\cal C}^0(\R^2,\R)$. Further we let
\[
\Omega(f) \ := \ \overline{\{f_{t,x}\mid (t,x)\in\R^2\}} \ .
\]
Then (\ref{e.ap-scalarODE}) induces a flow
\begin{equation}\label{e.ap-timespace}
\Phi \ : \ \R\times\Omega(f) \to \Omega(f)\quad , \quad (t,g) \mapsto
g_{t,\xi(t,g)} \ ,
\end{equation}
where $t\mapsto \xi(t,g)$ is the solution of $x'=g(t,x)$ with initial values
$t_0=x_0=0$. Again, Theorem~\ref{t.gotthead} can be applied (with $\omega_t(g) =
g_{t,0}$, $\varphi_t(g)= g_{0,t}$ and $\tau_t(g) = \xi(t,g)$) to determine
whether $\Phi$ is semiconjugate to the translation flow $(t,g) \mapsto
g_{t,t\rho}$ on $\Omega(f)$.  

\begin{cor}
  Suppose that that the translation flow $T^\rho : (t,g)\mapsto g_{t,t\rho}$ on
  $\Omega$ is minimal. Then the flow $\Phi$ given by (\ref{e.ap-timespace}) is
  semi-conjugate to $T^\rho$ if and only if the solution $t\mapsto \xi(t,f)$ of
  $x'=f(t,x)$ satisfies 
  \begin{equation}\label{e.11}
     \sup_{t\in\R}|\xi(t,f)| \ < \ \infty \ .
   \end{equation}
 \end{cor}
 Note that by continuity (\ref{e.11}) implies
 $\sup_{t\in\R,g\in\Omega(f)}|\xi(t,g)| < \infty$.
\medskip

 In order to give a somewhat more explicit example, suppose $F:\R^4\to\R$ is
 periodic with least period 1 in all coordinates, such that it projects to a
 function on the four-torus $\T^4=\R^4/\Z^4$. Fix a totally irrational vector
 $(\alpha,\beta)\in\R^4,\ \alpha,\beta\in\R^2$ and let
 $f(t,x)=F(t\cdot\alpha,x\cdot\beta)$. Then the hull $\Omega(f)$ is homeomorphic
 to $\T^4$, and the skew product flow $\Phi$ is of the form
\[
\Phi \ : \ \R\times\T^4 \to \T^4 \quad , \quad (t,z) \mapsto
(u+t\alpha,v+\tau_t(z)\beta) \ ,
\]
where $z=(u,v)$ and $\tau:\R\times\T^4$ satisfies
(\ref{e.tau-property}). Assuming that $\alpha=(1,\alpha_1)$, the Poincar\'e
section of this flow in $\{0\}\times\T^3$ gives rise to a skew product map of
the form
\begin{equation} \label{e.doubly_qpf_map}
\Psi \ : \ \T^3\to \T^3 \quad , \quad \zeta \mapsto
(\theta+\alpha_1,\xi+\tau_1(0,\xi)\beta) \ ,
\end{equation}
where $\zeta=(\theta,\xi)$ with $\theta\in\T^1$ and $\zeta\in\T^2$. One may say
that this is the simplest type of dynamical systems that exhibits
quasiperiodicity in both variables. Note that (the discrete-time flow generated
by) $\Psi$ can be written in the general form of (\ref{e.fdefinition}) with
$\omega_t(\zeta)=\zeta+t\cdot (\alpha_1,0,0)$, $\varphi_t(\zeta)=\zeta+t\cdot
(0,\beta)$ and $\tau_t(\zeta)=\tau_t(0,\zeta)$, $t\in\Z$.

We believe that maps of this type are a good starting point for studying further
questions concerning the type of flows introduced in Section 2, reflecting the
full generality of the setting while at the same time having a rather simple
structure. Hence, we end by pointing out the following still open problems.
\begin{questions} 
\alphlist
\item Does $\Psi$ always have a unique rotation number whenever $0$ is not
  contained in the rotation interval?
\item Suppose $\Psi$ is semiconjugate to a minimal rotation on $\T^3$. Does this
  imply that $\Psi$ is uniquely ergodic?
\item Suppose $\rho(\Phi)$ is $(\varphi,\omega)$-irrational and
  let $$\Phi^\eps_t(z) = (u+t\alpha,v+(\tau_t(z)+\eps)\beta) \ .$$ Then, is the
  map $\eps\mapsto \rho(\Phi^\eps)$ strictly increasing in $\eps=0$? ({\em
    Absence of mode-locking})

 For skew product circle flows, the answer
  is positive \cite{bjerkloev/jaeger:2009}.
\listend
\end{questions}

\bibliographystyle{unsrt}

\begin{thebibliography}{Kwa00}

\bibitem[AP10]{aliste:2010}
J.~Aliste-Prieto.
\newblock {Translation numbers for a class of maps on the dynamical systems
  arising from quasicrystals in the real line}.
\newblock {\em {Ergodic Theory Dynam. Systems}}, {30}({02}):565--594, {2010}.

\bibitem[BJ09]{bjerkloev/jaeger:2009}
K.~Bjerkl{\"o}v and T.~J{\"a}ger.
\newblock Rotation numbers for quasiperiodically forced circle maps --
  {M}ode-locking vs strict monotonicity.
\newblock {\em J.\ Am.\ Math.\ Soc.}, 22(2):353--362, 2009.

\bibitem[Ell69]{ellis:1969}
R.~Ellis.
\newblock {\em Lectures on topological dynamics}.
\newblock WA Benjamin, 1969.

\bibitem[Fav28]{favard:1928}
J.~Favard.
\newblock Sur les {\'e}quations diff{\'e}rentielles lin{\'e}aires a
  coefficients presque-p{\'e}riodiques.
\newblock {\em Acta Mathematica}, 51(1):31--81, 1928.

\bibitem[Gla03]{Glasner}
E.~Glasner.
\newblock {\em {Ergodic theory via joinings}}, volume {101} of {\em
  {Mathematical Surveys and Monographs}}.
\newblock {American Mathematical Society}, {Providence, RI}, {2003}.

\bibitem[Her83]{herman:1983}
M.~Herman.
\newblock Une {m\'{e}thode} pour minorer les exposants de {{L}yapunov} et
  quelques exemples montrant le {caract\`{e}re} local d'un {th\'{e}or\`{e}me}
  {d'Arnold} et de {Moser} sur le tore de dimension 2.
\newblock {\em Comm.\ Math.\ Helv.}, 58:453--502, 1983.

\bibitem[HY09]{huang/yi:2007}
W.~Huang and Y.~Yi.
\newblock Almost periodically forced circle flows.
\newblock {\em J. Func. Anal.}, 257(3):832--902, 2009.

\bibitem[Joh81]{johnson:1981}
R.~Johnson.
\newblock Bounded solutions of scalar, almost periodic linear equations.
\newblock {\em Illinois J. Math.}, 25(4):632--643, 1981.

\bibitem[JS06]{jaeger/stark:2006}
T.~J{\"a}ger and J.~Stark.
\newblock Towards a classification for quasiperiodically forced circle
  homeomorphisms.
\newblock {\em J.\ Lond.\ Math.\ Soc.}, 73(3):727--744, 2006.

\bibitem[KH97]{katok/hasselblatt:1997}
A.~Katok and B.~Hasselblatt.
\newblock {\em Introduction to the Modern Theory of Dynamical Systems}.
\newblock Cambridge University Press, 1997.

\bibitem[Kwa00]{kwapisz:2000}
J.~Kwapisz.
\newblock {Poincar{\'e} rotation number for maps of the real line with almost
  periodic displacement}.
\newblock {\em Nonlinearity}, 13(5):1841, 2000.

\bibitem[MR04]{miller:2004}
A.~Miller and J.~Rosenblatt.
\newblock Characterizations of regular almost periodicity in compact minimal
  abelian flows.
\newblock {\em Trans. Am. Math. Soc.}, 356(12):4909--4930, 2004.

\bibitem[SY98]{shen/yi:1998}
W.~Shen and Y.~Yi.
\newblock Almost automorphic and almost periodic dynamics in skew product
  semiflows.
\newblock {\em Mem.\ Am.\ Math.\ Soc.}, 136(647), 1998.

\end{thebibliography}

\end{document}